\documentclass{amsart}

\usepackage{amsmath,amssymb,latexsym,amsthm,graphicx}

\usepackage{mathrsfs}

\newtheorem{theorem}{Theorem}[section]

\newtheorem{lemma}[theorem]{Lemma}

\newtheorem{prop}[theorem]{Proposition}

\newtheorem{rem}[theorem]{Remark}

\newtheorem{coro}[theorem]{Corollary}

\newtheorem{defi}[theorem]{Definition}

\newcommand{\sgn}{\operatorname{sgn}}

\newcommand{\supp}{\operatorname{supp}}

\newcommand{\mS}{\mathcal{S}}

\newcommand{\mB}{\mathcal{B}}

\newcommand{\mC}{\mathcal{C}}

\newcommand{\mE}{\mathcal{E}}

\newcommand{\mR}{\mathcal{R}}

\newcommand{\mF}{\mathcal{F}}

\newcommand{\mD}{\mathcal{D}}

\newcommand{\bH}{\mathbb{H}}

\newcommand{\mL}{\mathcal{L}}

\newcommand{\mK}{\mathcal{K}}

\newcommand{\rN}{\mathbb{R}}

\newcommand{\mT}{\mathcal{T}}

\newcommand{\mO}{\mathcal{O}}

\newcommand{\uS}{\mathbb{S}}

\newcommand{\intl}{\int\limits}

\newcommand{\cl}{\operatorname{cl}}

\newcommand{\Llg}{\Lambda}

\newcommand{\ag}{\alpha}

\newcommand{\bg}{\beta}

\newcommand{\ve}{{\bf e}}

\newcommand{\sg}{\sigma}

\newcommand{\Og}{\Omega}

\newcommand{\ga}{\gamma}

\newcommand{\pdh}{\partial}

\newcommand{\cT}{\mathbb{T}}

\newcommand{\Cf}{\mathfrak{C}}

\title[Artifacts in Limited Data Tomography]{How Strong Are Streak Artifacts in Limited Angle Computed Tomography?}

\author{Linh V. Nguyen}

\thanks{The research is supported by the NSF grant DMS 1212125}

\address{Department of Mathematics, University of Idaho, Moscow, Idaho 83844, USA}
\email{lnguyen@uidaho.edu}

\begin{document}
\begin{abstract} In this article, we consider the limited angle problem in computed X-ray tomography. A common practice is to use the filtered back projection with the limited data. However, this practice has been shown to introduce the artifacts along some straight lines. In this work, we characterize the strength of these artifacts. 
\end{abstract}

\maketitle

\section{Introduction}\label{S:intro}
Computed X-ray tomography (CT) is probably the best known imaging modality. The object of interest is scanned with X-rays and the loss of intensity along the rays provides the data for the image reconstruction. Roughly speaking, the obtained data is the Radon transform $\mR f$ of the attenuation function $f$:
$$\mR f(\theta, s) = \int\limits_{\rN} f(s \, \theta + t \, \theta^\perp) \, dt,$$ where $s \in \rN$, $\theta \in \uS^1$, and $\theta^\perp$ is the unit vector $\frac{\pi}{2}$ counterclockwise from $\theta$. In order to reconstruct the image (i.e., the function $f$), one has to invert the Radon transform. This problem has attracted a significant amount of work in literature (see, e.g., \cite{Nat} and the reference therein).

\medskip

Let us briefly introduce the well known filtered-backprojection formula to invert $\mR$ \footnote{This work is motivated by \cite{FQ13} and we borrow the basic notations from there.}. We recall the $1$-dimensional Fourier transform $\mF$ and its inverse $\mF^{-1}$
\begin{eqnarray*}
\mF(g)(\tau) &=& \frac{1}{(2 \pi)^{1/2}} \int\limits_{\rN}  \, e^{-i \tau \, s} g(s) \, ds,\\
\mF^{-1}(g)(s)& =&\frac{1}{(2 \pi)^{1/2}} \int\limits_{\rN}  \, e^{i \tau \, s} g(\tau) \, d\tau.
\end{eqnarray*}
Let $g(\theta, s) \in \mS(\uS^1 \times \rN)$. The one dimensional Lambda operator is 
$$\Llg_s g= \mF^{-1}_\tau (|\tau| \, \mF_s g). $$
We also define the back-projection operator
$$\mR^*(g)(x) = \int\limits_{\uS^1} g(\theta, x \cdot \theta) \, d\theta.$$
Then, the well known filtered back-projection inversion formula for $\mR$ reads as (e.g., \cite{Nat})
\begin{equation} \label{E:Exact} f  = \mB \mR f := \frac{1}{4 \pi} \mR^* \Lambda_s \mR f.\end{equation}


\medskip

One major disadvantage of the above reconstruction formula is that it is not local (due to the non-locality of the operator $\Lambda_s$). That is, in order to find $f$ at a location $x$, one has to use the data at {\bf all} angles $\theta$ and distances $s$. Lambda tomography has been proposed to overcome this advantage. Namely, let us consider the following Lambda reconstruction formula
\begin{equation} \label{E:Lambda} \Lambda f := \mL \,\mR f := \frac{1}{4 \pi} \mR^*(-\pdh_s^2) \mR f .\end{equation}
This formula is local in the sense that in order to compute $\Lambda f(x)$, one only needs to use the data $\mR f$ in an arbitrarily small neighborhood of the set $$\{(\theta,s): \theta \cdot x = s\}.$$

\medskip

Although $\Lambda f$ is not equal to the image $f$, it is a good alternative for $f$, as follows. Let us recall the $2$-dimensional Fourier transform and its inverse
\begin{eqnarray*}
\mF(f)(\xi) &=& \hat{f}(\xi) = \frac{1}{2 \pi} \int\limits_{\rN^2}  \, e^{-i x \cdot \xi} f(x) \, dx,\\
\mF^{-1}(f)(x)& =& \check{f}(\xi) = \frac{1}{2 \pi} \int\limits_{\rN^2}  \, e^{i x \cdot \xi} f(\xi) \, d\xi.
\end{eqnarray*}
Then (see, e.g., \cite{smith1985mathematical})
\begin{equation} \label{E:Lam} \Lambda f  = \mF^{-1} \big(|\xi| \mF(f) \big). \end{equation}
That is (see Section \ref{S:pdo}), $\Lambda$ is a pseudo-differential operator of order one with the symbol $$\sg(x,\xi) \sim |\xi|, \mbox{ for large } |\xi|.$$  This, in particular, implies (see, e.g., \cite[Corollary 8.3.2]{Ho1-old})
$$WF(\Lambda f) = WF(f).$$ 
Here, $WF(f)$ and $WF(\Lambda f)$ are, respectively, the wave front set of $f$ and $\Lambda f$ (see Definition \ref{D:wave}). As explained in Section \ref{S:MA}, the wave front set can be used to describe the singularities of a function (or, more generally, a distribution). Moreover, (\ref{E:Lam}) implies that: $WF_s(f) = WF_{s-1}(\Llg f)$ (see Section~\ref{S:Sobolev} for more discussion). That is, Lambda reconstruction emphasizes the singularities (image features) by one order. In some applications, such as edge detection, this is an advantage of Lambda reconstruction over the exact reconstruction. Therefore, one may concentrate on finding $\Lambda f$ (which is simpler to compute) instead to $f$. More discussion about Lamda tomography can be found in, e.g., \cite{vainberg1981reconstruction,smith1985mathematical,Fari92,Fari97}. The reader is also referred to \cite{KLM,RamKat-Book} for other kinds of local tomography.

\medskip

In this article, we are interested in both exact construction formula \eqref{E:Exact} and the Lambda reconstruction formula \eqref{E:Lambda}. 

\medskip

Let us now turn our discussion to the main concern to this article, {\bf the limited angle problem} (see, e.g., \cite{RamKat-AML1,RamKat-AML2,KLM,Kat-JMAA,FQ13,Kuchment-Book}): \emph{$\mR f$ is only known for $(s,\theta) \in \rN \times \uS_\Phi$, for some $0<\Phi<\frac{\pi}{2}$.} Here, $\uS_\Phi \subsetneq \uS^1$ is defined by: 
$$\uS_\Phi=\{\theta \in \uS^1: \theta = \pm (\cos \phi, \sin \phi),\, |\phi| < \Phi\}.$$ 

\medskip

\noindent The reconstruction of $f$ from the limited data problem is severely ill-posed. Instead of trying to reconstruct the exact value of $f$, a common practice is to reconstruct the {\bf visible} singularities of $f$; which are all the elements $(x,\xi) \in WF(f)$ such that $\frac{\xi}{|\xi|} \in \uS_\Phi$. The reader is referred to \cite{FQ13} for more discussion about the visible singularities. 

\medskip

Let us define the following limited angle version of $\mR^*$:
$$\mR_\Phi^*(g)(x) = \int\limits_{\uS_\Phi} g(\theta, x \cdot \theta) \, d\theta,$$
and limited angle version of $\mB$ and $\mL$: 
\begin{eqnarray*} \mB_\Phi g &=& \frac{1}{4 \pi} \mR^*_\Phi \Lambda_s g,\\
\mL_\Phi \,g &=& \frac{1}{4 \pi} \mR_\Phi^* \big(-\frac{\pdh^2}{\pdh s^2} g \big).
\end{eqnarray*}

One can observe that $\mB_\Phi$ (or $\mL_\Phi$) is equal to applying $\mB$ (respectively $\mL$) to the limited data patched with zero outside the available range. It is shown in \cite{FQ13,Kat-JMAA} that $\mB_\Phi \mR$ and $\mL_\Phi \mR$ reconstruct the visible singularities of $f$. However, they also create added singularities (artifacts) into the picture. 
These artifacts are generated from the singularities of $f$ whose direction corresponds to the edges of $\uS_\Phi$ (i.e., singularities along the direction $\ve_1=(\cos \Phi, \sin \Phi)$ or $\ve_2=(\cos \Phi, -\sin \Phi)$). Their locations line up along straight lines orthogonal to their direction ($\ve_1$ or $\ve_2$) and, hence, called streak artifacts (the reader is referred to \cite{Kat-JMAA,FQ13} for detailed discussion).

\medskip

In this article, we characterize the strength of above mentioned artifacts. In fact, we will analyze those generated by the general operators $\mB_\Phi \mK \mR$ and $\mL_\Phi \mK \mR$. Here, $\mK$ is the operator that multiplies by $\kappa$,
$$\mK g(\theta,s) = \kappa(\theta) g(\theta,s),$$ where \begin{equation} \label{E:kappa} \kappa: \uS^1 \to \rN \mbox{ is a smooth even function such that } \kappa(\theta)>0 \mbox{ for all } \theta \in \uS_\Phi.\end{equation} Obviously, if $\kappa \equiv 1$ then $\mB_\Phi \mK \mR= \mB_\Phi \mR$ and $\mL_\Phi \mK \mR=\mL_\Phi  \mR$. Our results (see Theorem~\ref{T:Main1} and Section~\ref{S:Artifacts}), in particular, show that when $\kappa$ vanishes to order $k$ at the end points $(\pm \cos \Phi, \pm \sin \Phi)$ of $\uS_\Phi$, then the artifacts are reduced by $k$ orders.

\medskip

It is worth mentioning that the same problem has been studied in \cite{Kat-JMAA}. However, our approach and result are different from there. In particular, our result applies to general singularities, not only jumps.

\medskip

The article is organized as follows. In Section \ref{S:MA}, we introduce some basic concepts in microlocal analysis needed in this article. We then state the main result, Theorem~\ref{T:Main1}, and its consequences in Section \ref{S:Main}. Sections \ref{S:Model} and \ref{S:Proof} are dedicated to the proof of Theorem~\ref{T:Main1}. 

\section{Basic notions in microlocal analysis} \label{S:MA}

In this section we introduce several concepts in microlocal analysis.  We first discuss the definition of the wave front set and how to quantify it. We then provide some essential knowledge in pseudo-differential operators and Fourier integral operators, that is needed to understand the article.

Let $\Og \subset \rN^n$ be an open set. Throughout this article, we will denote by $\mD'(\Og)$ and $\mE'(\Og)$ the space of distributions and space of compactly supported distributions on $\Og$, respectively. 

\subsection{Wave front set} \label{S:wave}
Here is the definition of the wave front set of a distribution:
\begin{defi} [Wave Front Set \cite{hormander71fourier}]  \label{D:wave} Let $\Og \subset \rN^n$ be an open set, $u \in \mD'(\Og)$, and $(x_0,\xi_0) \in \cT^*\Og \setminus 0$ \footnote{$\cT^*\Og \setminus 0$ is the cotangent bundle of $\Og$ minus the zero section. It can be considered as $\Og \times (\rN^n \setminus 0)$.}. Then, $u$ is microlocally smooth at $(x_0,\xi_0)$ if there is a function $\varphi \in C_0^\infty(\Og)$ satisfying $\varphi(x_0) \neq 0$ and an open cone $V$ containing $\xi_0$, such that $\mF (\varphi f)$ is rapidly decreasing in $V$. That is, for any $N>0$,  there exists a constant $C_N$ such that
$$|\mF(\varphi u)(\xi)| \leq C_N (1+|\xi|)^{-N}, \mbox{ for all } \xi \in \rN^n.$$ The {\bf wave front set} of $u$, denoted by $WF(u)$, is the complement of the set of all $(x_0,\xi_0) \in \cT^* \Og$ where $u$ is microlocally smooth.
\end{defi} 

\medskip

\noindent An element $(x,\xi) \in WF(u)$ is called a singularity of $u$. The component $x$ indicates the location of the singularity, while $\xi$ indicates the direction of the singularity. For example, if $u$ is the characteristic function of an open set $\mO \Subset \Og$ with the smooth boundary $\pdh \mO$, then $(x,\xi) \in WF(u)$ if and only if $x \in \partial \mO$ and $\xi$ is perpendicular to the tangent plane of $\partial \mO$ at $x$. Detailed discussion can be found in \cite{Petersen-Book}.

\medskip

The study of the reconstruction of wave front set  in limited data X-ray and related transforms was initiated in \cite{Quinto-88,GrUDuke,GrU-Functional,GrUCon,QT93}. Similar study has become popular in many areas of imaging sciences.

\medskip

In this article, to study the relationship between $WF( \mB_\Phi \mK \mR f)$, $WF(\mL_\Phi \mK \mR f)$ and $WF(f)$, we will intensively use the following rule for calculus of wave front set (see, \cite[Theorem 2.5.14] {hormander71fourier}):

\begin{theorem} \label{T:wave-Ho} Let $\mT: \mE'(\rN^2) \to \mD'(\rN^2)$ be a continuous linear operator whose Schwartz kernel $\mu$ satisfies $WF(\mu)' \subset (\cT^* \rN^2 \setminus 0) \times (\cT^* \rN^2 \setminus 0)$. Then for any $f \in \mE'(\rN^2)$, $$WF(\mT) \subset WF(\mu)' \circ WF(f).$$
\end{theorem}
\noindent Here, $WF(\mu)'$ is the {\bf twisted wave front set} of $\mu$, defined by 
 $$WF(\mu)'=\{(x,\xi; y,- \eta):  (x,\xi;y,\eta) \in WF(\mu)\},$$
and, for any $A \subset \cT^* \rN^2 \times \cT^* \rN^2$,
$$A \circ WF(f):=\{(x,\xi): (x,\xi; y, \eta) \in A, \mbox{ for some } (y,\eta) \in WF(f)\} .$$

\subsection{Sobolev wave front set (singularities).} \label{S:Sobolev}
\noindent An important issue in imaging sciences is how to quantity the strength of a singularity. The following definition can be used to serve that purpose:
\begin{defi}[Sobolev Wave Front Set \cite{Petersen-Book}] \label{D:waves} Let $\Og \subset \rN^n$ be an open set, $u \in \mD'(\Og)$, and $(x_0,\xi_0) \in \cT^*\Og \setminus 0$. Then $u$ is in the space $H^s$ microlocally at $(x_0,\xi_0)$ if there is a function $\varphi \in C_0^\infty(\Og)$ satisfying $\varphi(x_0) \neq 0$ and a function $p(\xi)$ homogeneous of degree zero and smooth on $\rN^n \setminus 0$ with $p(\xi_0) \neq 0$, such that $$p(\xi) \, \mF (\varphi u)(\xi) \in L^2(\Og, (1+|\xi|^2)^s).$$
The $H^s$-wave front set of $u$, denoted by $WF_s(u)$, is the complement of the set of all $(x_0,\xi_0) \in \cT^* \Og$ where $u$ is not microlocally in the space $H^s$. 
\end{defi} 

The notion of Sobolev wave front set has been used in imaging sciences to indicate the strength of singularities (see, e.g., \cite{QT93,QuintoSONAR,Quinto-Rull-13}). We will use it to analyze the strength of the reconstructed singularities and artifacts generated by $\mT_m$. The reader should keep in mind that, roughly speaking, the smaller $s$ is, the {\bf rougher} (i.e., {\bf stronger}) a singularity $(x,\xi) \in WF_s(u)$ is. To compare two singularities $(x,\xi) \in WF(u_1)$ and $(y,\eta) \in WF(u_2)$ \footnote{The functions $u_1$, $u_2$ do not have to be defined on the same domain.}, one can make use of the following terminologies:
\begin{itemize}
\item[i)] $(x,\xi)$ and $(y,\eta)$ are of {\bf the same order}, if for all $s \in \rN$: $(x,\xi) \in WF_s(u_1)$ iff $(y,\eta) \in WF_s(u_2)$. 
\item[ii)] $(x,\xi)$ is {\bf stronger} than $(y,\eta)$, if there is $s \in \rN$ such that $(x,\xi) \in WF_s(u_1)$ but $(y,\eta) \not \in WF_s(u_2)$.
 \end{itemize}
Since for any $u \in \mD'(\Og)$ (see, e.g., \cite{Petersen-Book}) $$\bigcup_{s \in \rN} WF_s(u) = WF(u),$$ the above terminologies can be used to compare any two singularities. 
\medskip


\subsection{Conormal singularities.} \label{S:Co} In this section, we introduce a special kind of singularities, the conormal singularities, and how to quantify them. In many cases, concentrating on the conmornal singularities can be beneficial for the analysis. In-depth discussion of conormal singularities and their applications in inverse scattering and wave propagation can be found in \cite{GrU-Functional,Joshi,Suresh,deHoop-Uhlmann}). Its use in studying the artifacts of the x-ray transform in $\rN^3$ with sources on a curve was suggested in \cite{FLU}. Our introduction below only touches the surface of the topic and is designed to serve our presentation in Section~\ref{S:Main}.

\medskip

Assume that $S \subset \Og$ be a smooth surface of co-dimension $k$. Let $h \in C^\infty(\Og,\rN^k)$ be a defining function for $S$ with $rank(d h) = k$ on $S$. The class $I^r(S)$ consists of the distributions which locally can be written down as a finite sum of 
$$u(x) = \intl_{\rN^k} e^{i h(x) \cdot \theta} a(x,\theta) \, d\theta, $$
where $a \in S^r(\Og \times \rN^k)$ (see Section \ref{S:symbol} for the definition of $S^r(\Og \times \rN^k)$). We note that if $u \in I^r(S)$, then $WF(u) \subset N^* S$ (see, e.g, \cite{hormander71fourier}), where $N^*S$ is the conormal bundle of $S$.

\medskip

The order $r$ is a good indication for the strength of the singularities of $u \in I^r(S)$. For example, assume that $S$ is a smooth hypersurface in $\rN^n$ with non-vanishing Gaussian curvature. If $u$ is a smooth density on $S$, then $u \in I^0(S)$ (see, \cite[Sections 5.7]{Stein-93}). Meanwhile while, if $u$ has Heaviside-type (i.e., jump) singularities at $S$, then  $u \in I^{-1}(S)$ (see \cite[6.14]{Stein-93}). 

\medskip

In general, the {\bf smaller} $r$ is the {\bf smoother} (i.e., {\bf weaker}) the singularities of $u \in I^r(S)$ are.

\medskip

In this article, we are only interested in the case $S$ is a smooth curve in $\rN^2$. To fix our terminology, we introduce the following definition:
\begin{defi} Let $f \in \mD'(\rN^2)$. We say that $(x_0,\xi_0) \in WF(f)$ is a conormal singularity of order $r$ along the curve $S$ if there is $u \in I^r(S)$ such that
$$(x_0,\xi_0) \not \in WF(f-u). $$
\end{defi}
One can use the order $r$ to compare two conormal singularities $(x,\xi) \in WF(f_1)$ (along the curve $S_1$) and $(y,\eta) \in WF(f_2)$ (along the curve $S_2$), where $f_1,f_2$ are two distributions on $\rN^2$. For example, $(x,\xi)$ is {\bf weaker} than $(y,\eta)$, if there is $r \in \rN$ such that $(x,\xi)$ is of order $r$ while $(y,\eta)$ is not. 

\subsection{Pseudo-differential operators and Fourier integral operators}
In this section, we make a brief introduction to pseudo-differential operators and Fourier integral operators. The reader is referred to  \cite{hormander71fourier,Ho1-old,TrPseu,TrFour,Duistermaat}, among others, for a comprehensive introduction to the topics. The use of these operators to study the X-ray transform and its generalizations was introduced by Guillemin and Sternberg (see, e.g., \cite{Guillemin-Book,Guillemin-Sternberg}). It now becomes a standard technique in geometric integral transforms and tomography (see, e.g., \cite{Quinto-88,GrUDuke,GrU-Functional,GrUCon,QT93,louis00local}).

\medskip
One significant progress in microlocal analysis is theory of pseudo-differential operators with singular symbols, developed by Uhlmann, Guillemin, Melrose, and others (see, e.g., \cite{MU,GU,AnUhl}). The use of that theory to analyze the X-ray transform, when the canonical relation is not a local canonical graph, was pioneered by  Greenleaf and Uhlmann \cite{GrUDuke,GrUCon}. It has been then exploited intensively to analyze other imaging scenarios (e.g., \cite{FLU,NoChe,FeleaCPDE,FeleaQuinto,Suresh,FeleaSAR,Am-singular}). Although we do not directly use that theory in this article, it strongly influences our analysis.

\medskip

In the rest of this section, we do not attempt to provide the reader with an overview of pseudo-differential and Fourier integral operators. Instead, we only present the knowledge essential to understand our results.

\medskip

\subsubsection{\bf The symbol classes} \label{S:symbol} Let us start with the definition of the class $S^m(\Og \times \rN^N)$  of symbols (see, e.g., \cite{hormander71fourier}):
\begin{defi} Let $\Og \subset \rN^n$ be an open set. The space $S^m(\Og \times \rN^N)$ consists of all functions $a \in C^\infty(\Og \times (\rN^N \setminus 0))$ such that for any multi-indices $\ag,\bg$ and $K \Subset \Og$, there is a positive constant $C_{\ag,\bg,K}$ such that
\begin{equation} \label{E:sym-in} |\pdh_x^\ag \pdh_\xi^\bg a(x,\xi)| \leq C_{\ag,\bg,K} (1+ |\xi|)^{m -|\ag|}, \quad \mbox{ for all } (x,\xi)  \in K \times (\rN^N \setminus 0).\end{equation}
The elements of $S^m(\Og \times \rN^N)$ are called symbols of order $m$. 
\end{defi}

We will also denote $$S^{-\infty}(\Og \times \rN^N) = \bigcap_{m \in \rN} S^m(\Og \times \rN^N).$$
Let $a(x,\xi)$ and $a'(x,\xi)$ be two symbols. We write $a \sim a'$ if $a-a' \in S^{-\infty}(\Og \times \rN^N)$.

\medskip






If (\ref{E:sym-in}) is valid for $|\xi| \geq 1$, we say that {\bf $a \in S^m(\Og \times \rN^N)$ for large $|\xi|$}. We also define $a \sim a'$ for large $|\xi|$, in the same way as above \footnote{We will occasionally drop the term ``for large $|\xi|$" when it is clear from the context and not essential for the argument.}.

\subsubsection{\bf Pseudo-differential operators.} \label{S:pdo} 
Here is the definition of a pseudo-differential operator (see, e.g., \cite{Ho1-old,TrPseu,Duistermaat}):
\begin{defi}
Let $a(x,y,\xi) \in S^m((\rN^2 \times \rN^2) \times \rN^2)$. The operator $\mT: \mE'(\rN^2) \to \mD'(\rN^2)$ defined by $$\mT f(x)= \frac{1}{(2 \pi)^2} \intl_{\rN^2} \intl_{\rN^2} e^{i (x-y) \, \xi} a(x,y,\xi) f(y) dy \, d\xi$$
is called a pseudo-differential operator of order $m$ with the amplitude $a(x,y,\xi)$.
\end{defi}
 The oscillatory integral on the right hand side might not converge in the normal sense, even when $f \in C^\infty_0(\rN^2)$. The reader is referred to \cite{hormander71fourier,Ho1-old,TrPseu,Duistermaat} for its rigorous definition.

\medskip

If $\mu(x,y)$ is the Schwartz kernel of the above operator $\mT$, we write $\mu \in I^m(\Delta)$. Here, $ \Delta \subset (\mathbb{T}^*\rN^{2} \setminus 0) \times (\mathbb{T}^*\rN^{2} \setminus 0)$ is the diagonal relation $$\Delta = \{(x,\xi; x, \xi): (x,\xi) \in \cT^*\rN^{2} \setminus 0\}.$$
That is, $\mu \in I^m(\Delta)$ if
\begin{equation} \label{E:mu-pdo} \mu(x,y) = \frac{1}{(2 \pi)^2} \intl_{\rN^2} e^{i (x-y) \xi} \, a(x,y,\xi)  d\xi,\end{equation} where $a \in S^m((\rN^2 \times \rN^2) \times \rN^2)$.
The function $a(x,y,\xi)$ is also called the amplitude of $\mu$. 

\medskip

If $\mu \in I^m(\Delta)$, then (e.g., \cite{hormander71fourier}):
\begin{equation} \label{E:wave'} WF(\mu)' \subset \Delta.\end{equation}

\medskip

Consequently, due to Theorem~\ref{T:wave-Ho}, for any pseudo-differential operator $\mT$ \cite{hormander71fourier}:
\begin{equation} \label{E:inclu-pseu} WF(\mT f) \subset WF(f).\end{equation}
That is, a pseudo-differential operator does not generate new singularities. In several references, i.e. \cite{TrPseu,Petersen-Book}, the above inclusion is directly proved without explicitly employing Theorem~\ref{T:wave-Ho}.
\medskip

If $a(x,y,\xi)$ has the form $a(x,y,\xi)= \sg(x,\xi)$, then $a$ is called the {\bf symbol} of $\mT$ and $\mu$. The reader is referred to, e.g., \cite{hormander71fourier,TrPseu} for the definition and formula of the symbol of a pseudo-differential operator in the general case. However, in this article, we only need to know the symbol in the aforementioned special case.

\medskip

\begin{lemma} \label{L:PDO} Let $\mu$ be defined by 
\begin{equation*}\mu(x,y) = \frac{1}{(2 \pi)^2} \intl_{\rN^2} e^{i (x-y) \xi} \, a(x,\xi)  d\xi,\end{equation*}
where $a(x,\xi) \in S^{m}(\rN^2 \times \rN^2)$ for large $|\xi|$ such that $\partial_x^\ag \, a(x,\xi)$ is locally integrable with respect to $\xi$ for any orders $\ag$ and $x \in \rN^2$. Then, $\mu \in I^m(\Delta)$ with the symbol $\sg$ satisfying $\sg \sim a$ for large $|\xi|$. 
\end{lemma}

\begin{proof}
We can write $\mu =  \mu_0 + \mu_1$, where
$$\mu_0(x,y) = \frac{1}{(2 \pi)^2} \intl_{\rN^2} e^{i (x-y) \xi} c(|\xi|) \,a(x,\xi)  d\xi,$$
and
$$\mu_1(x,y) = \frac{1}{(2 \pi)^2} \intl_{\rN^2} e^{i (x-y) \xi} \big[1-c(|\xi|) \big] \, a(x,\xi)  d\xi.$$
Here, $c \in C^\infty(\rN)$ is such that $c(\tau) =0$ for $|\tau| \leq 1$ and $c(\tau)=1$ for $|\tau| \geq 2$.  

\medskip

Since $\partial_x^\ag \, a(x,\xi)$ is locally integrable with respect to $\xi$ for all $x \in \rN^2$, we obtain $\mu_1 \in C^\infty(\rN^2 \times \rN^2)$. Therefore, see e.g. \cite[Proposition 2.1]{TrPseu}, $\mu_1 \in I^{m}(\Delta)$ with the symbol  $\sg^1 \in S^{-\infty}(\rN^2 \times \rN^2)$. 

\medskip

On the other hand, $\mu_0 \in I^m(\Delta)$ with the symbol 

$$\sg^0(x,\xi) = c(|\xi|) \,a(x,\xi) \sim \, a(x,\xi), \quad \mbox{ for lage } |\xi|.$$

\medskip

Therefore, $\mu \in I^m(\Delta)$ with the symbol $\sg$ satisfying $\sg=\sg^0 + \sg^1 \sim a$ for large $|\xi|$.
\end{proof}
\medskip

Let $a(x,\xi) \in S^m(\rN^2 \times \rN^2)$. We say that $a(x,\xi)$ is elliptic near $(x^*,\xi^*)$ if there is a conic neighborhood $V$ of $(x^*,\xi^*)$ and positive numbers $C,\rho$ such that
$$|a(x,\xi)| \geq C(1+|\xi|)^m, \quad \mbox{ for all } (x,\xi) \in V, \mbox{ satisfying } |\xi| \geq \rho.$$
We also say that a pseudo-differential operator (of order $m$) $\mT$ is elliptic near $(x^*,\xi^*)$ if its symbol is elliptic near $(x^*,\xi^*)$.

\medskip

The following result tells us the effect of an elliptic operator on the singularity at $(x^*,\xi^*)$ (see, e.g., \cite{TrPseu,Petersen-Book}):
\begin{theorem} \label{T:Pet}
Let $\mT: \mE'(\rN^2) \to \mD'(\rN^2)$ be a pseudo-differential operator of order $m$. Assume that $\mT$ is elliptic near $(x^*,\xi^*)$. Then, for any $f \in \mE'(\rN^2)$ and $s \in \rN$, \begin{equation*} (x^*,\xi^*) \in WF_s(f) \mbox{ if and only if } (x^*,\xi^*) \in WF_{s-m}(\mT f).\end{equation*}
\end{theorem}

\medskip

The following technical term will be used in the statement of Theorem~\ref{T:Main1}~a):
\begin{defi} \label{D:PDO}
Let $A \subset \Delta$ be a conic set that is open in the topology of $\Delta$, induced from $(\cT^* \rN^2 \setminus 0) \times (\cT^* \rN^2 \setminus 0)$. We say that {\bf near $A$, $\mu$ is microlocally in the space $I^m(\Delta)$ with the symbol $\sg$} if the following holds: for each element $(x^*,\xi^*;x^*,\xi^*) \in A$ there exists $\mu_* \in I^m(\Delta)$ such that $$(x^*,\xi^*; x^*, \xi^*) \not \in WF(\mu-\mu_*)',$$ and the symbol of $\mu_*$ is equal to $\sg(x,\xi)$ in a conic neighborhood of $(x^*,\xi^*)$. 
\end{defi}

\medskip
In Section~\ref{S:S}, we will need the following more ``microlocalized" version of Theorem~\ref{T:Pet}: 
\begin{coro} \label{C:Pet}
Let $\mT: \mE'(\rN^2) \to \mD'(\rN^2)$ be a linear operator whose Schwartz kernel $\mu \in \mD'(\rN^2 \times \rN^2)$ satisfies $WF(\mu)' \subset (\cT^*\rN^2\setminus 0)\times (\cT^*\rN^2\setminus 0)$ and near $A \subset \Delta$ \footnote{Here, $A$ satisfies the condition in Definition~\ref{D:PDO}.}, $\mu$ is microlocally in $I^m(\Delta)$ with the symbol $\sg(x,\xi)$. Assume that $(x^*,\xi^*;x^*,\xi^*) \in A$, $\sg(x,\xi)$ is elliptic of order $m$ near $(x^*,\xi^*)$, and \begin{equation} \label{E:wave-x*} \{(x^*,\xi^*;y,\eta) \in WF(\mu)': (y,\eta) \in WF(f)\} \subset \Delta.\end{equation} Then, for any $s \in \rN$, \begin{equation*} (x^*,\xi^*) \in WF_s(f) \mbox{ if and only if } (x^*,\xi^*) \in WF_{s-m}(\mT f).\end{equation*}
\end{coro}
We provide its proof here to illuminate the use of the above assumptions.
\begin{proof} Since $(x^*,\xi^*;x^*,\xi^*) \in A$ and $\mu$ is microlocally in $I^m(\Delta)$ near $A$ with symbol $\sg$, there is $\mu_* \in I^m(\Delta)$ such that \begin{equation} \label{E:xx} (x^*,\xi^*; x^*,\xi^*) \not \in WF(\mu-\mu_*).\end{equation}
Moreover, the symbol of $\mu_*$ is $\sg(x,\xi)$ near $(x^*,\xi^*)$. Let $\mT_*$ be the pseudo-differential operator of order $m$ whose Schwartz kernel is $\mu_*$. Then, $\mT_*$ is elliptic near $(x^*,\xi^*)$. Applying Theorem~\ref{T:Pet}, we obtain, for any $s \in \rN$, \begin{equation*} (x^*,\xi^*) \in WF_s(f) \mbox{ if and only if } (x^*,\xi^*) \in WF_{s-m}(\mT_* f).\end{equation*}
Therefore, it suffices to prove that \begin{equation} \label{E:contra} (x^*, \xi^*) \not \in WF(\mT f - \mT_* f).\end{equation}

\medskip

Let us proceed to prove (\ref{E:contra}). We observe that
\begin{equation} \label{E:wave-mumu} WF(\mu - \mu_*)' \subset WF(\mu)' \cup WF(\mu_*)' \subset WF(\mu)' \cup \Delta.\end{equation}
Here, to obtain the second inclusion, we have used (\ref{E:wave'}) for $WF(\mu_*)'$.

\medskip

The inclusion (\ref{E:wave-mumu}), together with $WF(\mu)' \subset (\cT^*\rN^2 \setminus 0) \times (\cT^* \rN^2\setminus 0)$, implies
$$WF(\mu-\mu_*)'  \subset (\cT^*\rN^2 \setminus 0) \times (\cT^* \rN^2\setminus 0). $$
Hence, due Theorem~\ref{T:wave-Ho},
\begin{equation}\label{E:mu-mu'}
WF(\mT f -\mT_* f) \subset WF(\mu-\mu_*)' \circ WF(f).
\end{equation}

\medskip

We now prove (\ref{E:contra}) by contradiction. To that end, let us assume $$(x^*,\xi^*) \in WF(\mT f - \mT^* f).$$ Then, from (\ref{E:mu-mu'}), there is $(y,\eta) \in WF(f)$ such that \begin{equation} \label{E:xy} (x^*,\xi^*;y,\eta) \in WF(\mu-\mu_*)'.\end{equation} From (\ref{E:wave-mumu}), we obtain 
$$(x^*,\xi^*;y,\eta) \in WF(\mu)'  \cup \Delta.$$
From (\ref{E:wave-x*}), we arrive to
$$(x^*,\xi^*; y,\eta) \subset \Delta.$$
That is $(x^*,\xi^*; y,\eta) = (x^*,\xi^*; x^*,\xi^*)$. We, hence, arrive to a contradiction between (\ref{E:xy}) and (\ref{E:xx}). This finishes our proof.
\end{proof}

\subsubsection{\bf Fourier integral operators (FIOs).} \label{S:FIO} In this section, we do not attempt to give the general definition and properties of FIOs. We, instead, introduce a special type of FIOs that is needed in this article. The interested reader is referred to \cite{hormander71fourier,Ho1-old,TrFour,Duistermaat} for a comprehensive presentation on FIOs.

\medskip

Let $\ve \in \rN^2$ be a unit vector and $a(x,y,\tau) \in S^{m+ \frac{1}{2}}((\rN^2 \times \rN^2) \times \rN)$. Then, the operator $\mT: \mE'(\rN^2) \to \mD'(\rN^2)$ defined by
\begin{equation} \label{E:FIO} \mT(f)(x) =\frac{1}{(2 \pi)^{3/2}} \intl_{\rN^2} \intl_{\rN} e^{i (x-y) \cdot \ve \, \tau} a(x,y,\tau) f(y) d\tau \, dy \end{equation}
is a Fourier integral operator of order $m$ with the phase function $\phi(x,y,\tau)= (x-y) \cdot \ve \, \tau$ and amplitude function $a(x,y,\tau)$. The integral on the right hand side of (\ref{E:FIO}) may not converge in the normal sense, even if $f \in C^\infty_0(\rN^2)$. The reader is referred to, e.g., \cite{hormander71fourier,TrFour} for its rigorous definition.

\begin{rem} \label{R:Rule} One may notice the difference between the order of $\mT$ and that of the amplitude $a(x,y,\tau)$. This comes from the following general rule (see, e.g., \cite{hormander71fourier,TrFour})
\begin{equation} \mbox{ order of } \mT = \mbox{ order of } a +(N-n)/2.\end{equation}
Here, $n=n_x=n_y$ is the dimension of $x$ and $y$, and $N$ is the dimension of $\tau$. In our case, $n=2$ and $N=1$.
\end{rem}

\medskip

Let $\mC \subset (\cT^* \rN^2 \setminus 0) \times (\cT^* \rN^2 \setminus 0)$ be defined by
\begin{eqnarray*}
\mC = \{(x,\ga\, \ve; x + t \ve^\perp, \ga\, \ve): x \in \rN^2,~\ga,t \in \rN,~\ga \neq0\}.
\end{eqnarray*}
Then, $\mC$ is called the canonical relation of $\mT$ \footnote{The reader is referred to \cite{hormander71fourier} for the canonical relation of a general FIO.}. The following result tells us how $\mT$ transforms the wavefront set of a distribution $f \in \mE'(\rN^2)$ (see, e.g., \cite{hormander71fourier,TrFour}):
\begin{equation} \label{E:incl} WF(\mT f) \subset \mC \circ WF(f).\end{equation}

\medskip

Let  $\mu$ be the Schwartz kernel of $\mT$. That is, 
\begin{equation} \label{E:kFIO} \mu(x,y) =\frac{1}{(2 \pi)^{3/2}} \intl_{\rN} e^{i (x-y) \cdot \ve \, \tau} a(x,y,\tau) d\tau.\end{equation} Then, we write $\mu \in I^m(\mC)$ (the interested reader is referred to \cite{hormander71fourier} for the definition of the H\"omander space $I^m(\mC)$ where $\mC$ is a general Lagrangian). 

\medskip

For any $\mu \in I^m(\mC)$, one has \footnote{The inclusion (\ref{E:mu-FIO}), in fact, implies (\ref{E:incl}), due to Theorem~\ref{T:wave-Ho}.}
\begin{equation} \label{E:mu-FIO} WF(\mu) ' \subset \mC.\end{equation}
Here, again, $WF(\mu)'$ is the twisted wave front set of $\mu$, defined by 
 $$WF(\mu)'=\{(x,\xi; y,- \eta):  (x,\xi;y,\eta) \in WF(\mu)\}.$$

\medskip

Assume that $a_{m+\frac{1}{2}}(x,y,\tau) \in C^\infty((\rN^2 \times \rN^2) \times (\rN \setminus 0))$ is homogeneous of degree $m+\frac{1}{2}$ \footnote{That is, for all $s>0$, $a_m(x,y,s \, \tau)= s^{m+\frac{1}{2}} a_m(x,y, \tau)$.} and not identically zero on the projection $\mC_{x,y}$ of $\mC$ on the $(x,y)$-space such that
$$a - a_{m+\frac{1}{2}} \in S^{m-\frac{1}{2}}((\rN^2 \times \rN^2) \times \rN), \mbox{ for large } |\tau|.$$
Then, we say that $\sg(x,y,\tau) = a_{m+\frac{1}{2}}(x,y,\tau)|_{(x,y) \in \mC_{x,y}}$ is the {\bf principal symbol} of $\mu$ associated with the phase function $\phi(x,y,\tau) = (x-y) \cdot \ve \, \tau$. The rigorous definition of the principal symbol of a general FIO is quite complicated and abstract. The interested reader is referred \cite{hormander71fourier} for the matter.

\begin{lemma} \label{L:FIO} Let $\mu$ be defined as in (\ref{E:kFIO}), where $a(x,y,\tau) \in S^{m+\frac{1}{2}}((\rN^2 \times \rN^2) \times \rN)$ {\bf for large} $|\tau|$ such that $\partial_x^\ag \, \partial_y^\bg a(x,y,\tau)$ is locally integrable with respect to $\tau$ for all $(x,y) \in \rN^2 \times \rN^2$. Then, $\mu \in I^m(\mC)$ with the amplitude $a'(x,y,\tau)\sim a(x,y,\tau)$ for large $|\tau|$.
\end{lemma}
The proof of Lemma~\ref{L:FIO} is similar to that of Lemma~\ref{L:PDO}. We skip it for the sake of brevity.

\medskip

The following result is helpful to analyze the strength of the artifacts in terms of their Sobolev order:
\begin{theorem} \label{T:Ho}
Let $\mT$ be defined in (\ref{E:FIO}) and $f \in \mE'(\rN^2)$. Assume that $(x^*,\xi^*=\ga_* \ve)  \in WF_{s}(\mT f)$. Then, $(y^*=x^*+t_0 \ve^\perp,\xi^*) \in WF_{s+m+\frac{1}{2}}(f)$ for some $t_0 \in \rN$. 
\end{theorem}
Theorem~\ref{T:Ho} comes from (\ref{E:incl}) and the continuity of the FIOs between Sobolev spaces \cite[Theorem 4.3.2]{hormander71fourier}. The reader should notice the order $(s+m+\frac{1}{2})$ (instead of $s+m$) in the conclusion. This is due to the fact that $\mC$ is not a local canonical graph \footnote{The reader is referred to \cite[Definition 4.1.5]{hormander71fourier} for the definition of a local canonical graph.}, and that order appears when applying \cite[Theorem 4.3.2]{hormander71fourier}.

\medskip

The following technical term will be used in the statement of Theorem~\ref{T:Main1}~b):
\begin{defi} \label{D:FIO} Let $A \subset \mC$ be a conic set that is open in the topology of $\mC$, induced from $(\cT^* \rN^2 \setminus 0) \times (\cT^* \rN^2 \setminus 0)$. We say that {\bf near $A$, $\mu$ is microlocally in the space $I^m(\mC)$ with the principal symbol $\sg_0(x,y,\tau)$} if the following holds: for each element $(x^*,\xi^*;y^*,\eta^*) \in A$ there exists $\mu_* \in I^m(\mC)$ such that $$(x^*,\xi^*; y^*, \eta^*) \not \in WF(\mu-\mu_*)'$$ and the principal symbol of $\mu_*$ is equal to $\sg_0(x,y,\tau)$ for all $(x,y)$ in a neighborhood of $(x^*,y^*)$.  
\end{defi}

In Section \ref{S:Artifacts}, we will in fact  need the following more ``microlocalized" version of Theorem \ref{T:Ho}:

\begin{coro} \label{C:Ho}
Let $\mT: \mE'(\rN^2) \to \mD'(\rN^2)$ be a linear operator whose Schwartz kernel $\mu \in \mD'(\rN^2 \times \rN^2)$ satisfies $WF(\mu) \subset (\cT^*\rN^2\setminus 0)\times (\cT^*\rN^2\setminus 0)$ and near $A \subset \mC$ \footnote{Here, $A$ satisfies the condition in Definition~\ref{D:FIO}.}, $\mu$ is microlocally in $I^m(\mC)$. Assume that $(x^*,\xi^*=\ga_* \ve) \in WF_s(\mT f)$ and $$\{(x^*,\xi^*;y,\eta) \in WF(\mu)' \cup \mC: (y,\eta) \in WF(f)\} \mbox{ is a compact subset of  A}.$$  Then, $(y^*=x^*+t_0 \ve^\perp,\xi^*) \in WF_{s+m+\frac{1}{2}}(f)$ for some $t_0 \in \rN$
\end{coro}
Corollary~\ref{C:Ho} can be proved in the same manner as Corollary~\ref{C:Pet}, where Theorem~\ref{T:Ho} is used in place of Theorem~\ref{T:Pet}. We skip it for the sake of brevity. 

\medskip
The following result is useful to analyze the artifacts when the original singularities are conormal:
\begin{theorem} \label{T:Spread}
Suppose that all the assumptions in Corollary~\ref{C:Ho} hold. Assume further that:
\begin{itemize}
\item [1)] There are at most finitely many $y^* \in \rN^2$ such that $y^* = x^* + t \ve^\perp$ for some $t \in \rN$ and $(y^*, \xi^*) \in WF(f)$. 
\item [2)] For each such $y^*$, $(y^*,\xi^*)$ is a conormal singularity of order $r$ along a curve $S$ which has nonzero curvature at $y^*$.
 \end{itemize}
Then, $(x^*,\xi^*=\ga_* \ve)$ is a conormal singularity of order at most $m+r$ along the line $$\ell=\{y \in \rN^2: y= x^* + t \ve^\perp, t \in \rN\}.$$
\end{theorem}

Theorem~\ref{T:Spread} can be proved in the same manner as Corollary~\ref{C:Ho}, where Theorem~\ref{T:Ho} is replaced by \cite[Proposition 2.1]{GrU-Functional} \footnote{In this reference, the authors consider the mapping property of a class of pseudo-differential operators with the singular symbols between spaces of conormal distributions.}. We skip the details for the sake of brevity.


\section{Statement of main result, its interpretation, and organization of the proof}\label{S:Main}
Let us denote by $W_\Phi$ the polar wedge
$$W_\Phi= \rN_* \cdot \uS_\Phi =\{r \, \theta: r \neq 0, \,\theta \in \uS_\Phi \},$$
and $\chi_{\Phi}$ the characteristic function of its closure $\cl(W_\Phi)$. \\

Assume that $m$ is a nonnegative real number and $\kappa \in C^\infty(\uS^1)$ satisfies (\ref{E:kappa}). Let $\mT_m$ be the linear operator whose Schwartz kernel is:
\begin{equation} \label{E:sch} \mu_m(x,y) = \frac{1}{(2 \pi)^2}\intl_{\rN^2} e^{i (x-y) \cdot \xi } |\xi|^m \, \kappa (\xi/|\xi|)\, \chi_{\Phi}(\xi) \, d \xi.\end{equation}
It can be easily shown that (see, e.g., \cite{FQ13}): $ \mT_0 = \mB_\Phi \mK \mR$ and $ \mT_1 = \mL_\Phi \mK \mR$. Moreover, $\mT_0 = \mB_\Phi \mR$ and $ \mT_1 = \mL_\Phi \mR$, if $\kappa \equiv 1$. 

\medskip

From now on, we will study the general operator $\mT_m$ for all real numbers $m \geq 0$. Our results will translate naturally to $\mB_\Phi \mR$, $\mL_\Phi \mR$, $\mB_\Phi \mK \mR$, and $\mL_\Phi \mK \mR$.

\medskip

Assume that $\kappa$ vanishes to {\bf infinite} order at the boundary points of $\uS_\Phi$. Then  
$$a(x,\xi):=|\xi|^m \, \kappa(\xi/|\xi|) \, \chi_{\Phi}(\xi)$$ is smooth on $\rN^2 \times (\rN^2 \setminus 0)$.
Therefore, $\mT$ is a pseudo-differential operator with the symbol $\sg(x,\xi) \sim a(x,\xi)$ for large $|\xi|$ (see Lemma~\ref{L:PDO}). This, in particular, implies that $WF(\mT_m f) \subset WF(f)$ (see Section~\ref{S:pdo}). That is, $\mT_m$ does not generate artifacts. The reader is referred to \cite{FQ13} for detailed arguments. We do not analyze this case any further in this article.

\medskip

We now concentrate on the case $\kappa$ only vanishes to {\bf finite} order (or does not vanishes at all, as in the case of $\mB_\Phi \mR$ and $\mL_\Phi \mR$) at the boundary points of $\uS_\Phi$. Then, $a(x,\xi)$ is no longer smooth on $\rN^2 \times (\rN^2 \setminus 0)$ with respect to the variable $\xi$. Therefore, $\mT_m$ is not a pseudo-differential operator in the standard sense. It is, instead, a pseudo-differential operator with a singular symbol. As we will show later in Theorem~\ref{T:Main1}, the twisted wave front set of $\mu_m$ is contained in the union of three Lagrangians $\Delta$, $\mC_1$, $\mC_2$. The part of the twisted wave front of $\mu_m$ in $\Delta$, the diagonal, is responsible for the reconstruction of singularities. Those in $\mC_1, \mC_2$ are responsible for the generation of artifacts. By analyzing their strength (i.e., order), see Theorem~\ref{T:Main1}~a)\&b), we can describe the strength of the reconstructed singularities and artifacts, see Sections~\ref{S:S} and \ref{S:Artifacts}. 

\medskip

We now describe our results in details. Let $ \Delta \subset (\mathbb{T}^*\rN^{2} \setminus 0) \times (\mathbb{T}^*\rN^{2} \setminus 0)$ be the diagonal relation $$\Delta = \{(x,\xi; x, \xi): (x,\xi) \in \cT^*\rN^{2} \setminus 0\},$$
and $$\Delta_\Phi=\{(x,\xi; x, \xi)\in \Delta: \xi \in \cl(W_\Phi)\}.$$

For $j=1,2$, $ \mC_j \subset (\mathbb{T}^*\rN^{2} \setminus 0) \times (\mathbb{T}^*\rN^{2} \setminus 0)$ is defined by
\begin{eqnarray*}
\mC_j = \{(x,\ga\, \ve_j; x + t \ve_j^\perp, \ga\, \ve_j): x \in \rN^2,~\ga,t \in \rN,~\ga \neq 0\}.
\end{eqnarray*}
We recall that, as mentioned in the introduction, $$\ve_1 = (\cos \Phi, \sin \Phi) \mbox{ and } \ve_2 = (\cos \Phi, -\sin \Phi).$$


\medskip

Here is the main result of this article:
\begin{theorem}\label{T:Main1}
We have \footnote{The reader is referred to Section~\ref{S:wave} for the definition of the twisted wave front set $WF(\mu_m)'$.} \begin{equation}\label{E:wave-in} WF(\mu_m)' \subset \Delta_\Phi \cup \mC_1 \cup \mC_2.\end{equation}
Furthermore,
\begin{itemize}
\item[a)] Near $\Delta \setminus (\mC_1 \cup \mC_2)$, $\mu_m$ is microlocally in the space $I^{m}(\Delta)$ with the symbol 
$$\sg(x,\xi) \sim |\xi|^m \,\kappa(\xi/|\xi|) \, \chi_{\Phi}(\xi), \quad \mbox{ for lage } |\xi|.$$
\item[b)] Let $\varkappa:[0,2 \pi] \to \rN$ be defined by $\varkappa(\phi) = \kappa(\cos \phi, \sin \phi)$. Assume that $\varkappa$ vanishes to order $k$ at $\phi = \pm \Phi$ \footnote{That is, $\varkappa^{(k)}(\pm \Phi) \neq 0$ and $\varkappa^{(l)}(\pm \Phi)=0$, for all $0 \leq l \leq k-1$.}. Then, near $\mC_j \setminus \Delta$, $\mu_m $ is microlocally in the space $I^{m-k-1/2}(\mC_j)$. Moreover, given the phase function $\phi_j(x,y,\tau) = (x-y) \cdot \ve_j \, \tau,$ its principal symbol is
\begin{eqnarray*}
\sg_0(x, y = x+ t \, \ve_j^\perp, \tau) = \frac{(-1)^j}{\sqrt{2 \pi}}\frac{\varkappa^{(k)} ((-1)^{j+1} \Phi)}{(i\,  \, \sgn (\tau) \, t)^{k+1}}  \, |\tau|^{m-k}.
\end{eqnarray*}
\end{itemize}
\end{theorem}



\medskip

\begin{rem} \label{R:order} We note that the order $(m-k -\frac{1}{2})$ of $\mu_m$ on $\mC_j \setminus \Delta$, stated in Theorem~\ref{T:Main1}~b), follows from the rule stated in Remark \ref{R:Rule}.
\end{rem}

The reader is referred to Definitions~\ref{D:PDO} \& \ref{D:FIO} for the technical terms used in the statement of Theorem~\ref{T:Main1}~a)~\&~b). We now use Theorem~\ref{T:Main1} (together with Theorem~\ref{T:wave-Ho}, Corollary~\ref{C:Pet}, Corollary~\ref{C:Ho}, and Theorem~\ref{T:Spread}) to analyze the reconstruction of singularities and generation of artifacts due to $\mT_m f$. In the below discussion, we assume that $f \in \mE'(\rN^2)$.

\medskip

Let us start with the the inclusion (\ref{E:wave-in}). Applying Theorem~\ref{T:wave-Ho}, we obtain
\begin{eqnarray*} WF(\mT_m f) \subset WF(\mu_m)' \circ WF(f).\end{eqnarray*}
Therefore,
\begin{equation} \label{E:wave-mT}WF(\mT_m f) \subset \big[\Delta_\Phi \circ WF(f) \big] \cup \big[\mC_1 \circ WF(f) \big]  \cup \big[\mC_2 \circ WF(f) \big]. \end{equation}
The first part on the right hand side is the set of all possible reconstructed singularities of $f$; meanwhile, the other two parts contain all the possible artifacts. Let us now analyze them in more details.

\subsection{\bf Reconstruction of singularities. } \label{S:S} 
From the inclusion (\ref{E:wave-mT}), the set of possible reconstructed singularities is $$\Delta_\Phi \circ WF(f) = \{(x,\xi) \in WF(f): \xi \in \cl(W_\Phi)\}.$$
Therefore, $\mT_m$ does not reconstruct $(x^*,\xi^*) \in WF(f)$ such that $\xi^* \not \in \cl(W_\Phi)$. Such a singularity $(x^*,\xi^*)$ of $f$ is called invisible.

\medskip

On the other hand, assume that $(x^*,\xi^*) \in WF(f)$ and $\xi^* \in W_\Phi$, i.e., $(x^*,\xi^*)$ is a {\bf visible} singularity (see, e.g., \cite{FQ13}). From Theorem~\ref{T:Main1}~a) and  Corollary~\ref{C:Pet},  $$(x^*,\xi^*) \in WF_{s-m} (\mT_m f) \mbox{ if and only if } (x^*,\xi^*) \in WF_s(f).$$ That is, $\mT_m$ reconstructs the singularity of $f$ at $(x^*,\xi^*)$; and the reconstructed singularity is $m$ order(s) stronger than the original singularity. In particular, the visible singularities are reconstructed with the same order if using $\mB_\Phi \mR$ and $\mB_\Phi \mK \mR$. Meanwhile, the visible singularities are emphasized by one order if using $\mL_\Phi \mR$ and $\mL_\Phi \mK \mR$ for the reconstruction.

\subsection{\bf Generation of artifacts.} \label{S:Artifacts} We notice that  $$\mC_j \circ WF(f) = \{(x,\ga\, \ve_j):  (y = x + t \ve_j^\perp, \ga\, \ve_j)  \in WF(f),~ \mbox{for some } t \in \rN , \ga \neq 0\}.$$

Assume that $(x^*,\xi^*)$ is an artifact. That is, $(x^*, \xi^*) \not \in WF(f)$ and $(x^*,\xi^*) \in WF(\mT_m f)$. From (\ref{E:wave-mT}), we obtain $$(x^*,\xi^*) \in \big[(\mC_1 \setminus \Delta) \circ WF(f)\big] \cup \big[(\mC_2 \setminus \Delta) \circ WF(f) \big].$$ 
Therefore, there are $\ga_* \neq 0$ and $j=1$ or $2$ such that $\xi^* = \ga_* \ve_j$. Moreover, there is at least one point $y^* \in \rN^2$ such that \begin{equation} \label{E:co-sing} y^* = x^*+ t \ve_j^\perp, \mbox{ for some } t \neq 0, \mbox{ and } (y^*,\xi^*= \ga_* \ve_j) \in WF(f).\end{equation} Each such $(y^*,\xi^*) \in WF(f)$ is called a singularity {\bf corresponding to (or generating)}  $(x^*,\xi^*)$. Due to Theorem \ref{T:Main1}~b) and  Corollary~\ref{C:Ho}, we obtain that if $(x,\xi) \in WF_{s}(\mT_m f)$ then at least one of its corresponding singularity $(y^*,\xi^*)$ satisfies $(y^*,\xi^*) \in WF_{s+(m-k)}(f)$. That is, using the Sobolev order to indicate the strength, we conclude:
\begin{itemize}
\item[S.1)] The artifacts are {\bf at most} $(m-k)$ order(s) stronger than their strongest generating singularities if $m>k$.
\item[S.2)] The artifacts are {\bf at most} as strong as their strongest generating singularities if $m=k$.
\item[S.3)] The artifacts are {\bf at least} $(k-m)$ order(s) smoother than their strongest generating singularities if $k>m$.
\end{itemize}
Let us assume further that:
\begin{itemize}
\item[A.1)] $(x^*,\xi^*)$ has only finitely many generating singularities $(y^*,\xi^*) \in WF(f)$, and
\item[A.2)] each such generating singularity is conormal of order at most $r$ along a curve $S$ which has nonzero curvature at $y^*$. 
\end{itemize}
Then, due to Theorem~\ref{T:Spread}, $(x^*,\xi^*)$ is a conormal singularity of order at most $(r+m-k - \frac{1}{2})$ along the curve
$$\ell=\{y: y=x^* + t \ve^\perp,~t \in \rN\}.$$
 That is, using the order of conormal singularity to indicate the strength, we conclude for such artifacts $(x^*,\xi^*)$:
\begin{itemize}
\item[S.1')] The artifacts are {\bf at most} $(m-k-\frac{1}{2})$ order(s) stronger than their strongest generating singularities if $m>k+\frac{1}{2}$.
\item[S.2')] The artifacts are {\bf at most} as strong as their strongest generating singularities if $m=k+\frac{1}{2}$,  
\item[S.3')] The artifacts are {\bf at least} $(k+\frac{1}{2}-m)$ order(s) smoother than their strongest generating singularities if $m<k+\frac{1}{2}$. 
\end{itemize}

\medskip

The descriptions in (S.1-3) and (S1'-3') can be easily interpreted for $\mB_\Phi \mR$, $\mL_\Phi \mR$, $\mB_\Phi \mK \mR$, and $\mL_\Phi \mK \mR$, by plugging the corresponding values of $m$ and $k$.

\medskip

In particular, for the case of $\mB_\Phi \mR$ (i.e., $m=k=0$), using (S.2), we obtain that the artifacts are at most as strong as their strongest generating singularities. If we assume that (A.1) and (A.2) hold, then due to (S.3'), the artifacts are (at least) half an order {\bf weaker} than the strongest generating singularities. This fact can be observed in \cite[Figure 1]{FQ13}, where the artifacts are indeed visually weaker than the generating singularities.

\subsection{Main ideas and structure of the proof of Theorem \ref{T:Main1}}
Let us briefly discuss the main ideas of the proof of Theorem \ref{T:Main1}. The proof of (\ref{E:wave-in}) follows from the approach in \cite{FQ13}, which uses the relationship between the wave front set of a homogeneous distribution and that of its Fourier transform \cite[Theorem  8.1.4]{Ho1-old}. To prove parts a) and b), we use a partition of unity to decompose the integral (\ref{E:sch}) into five parts. The first one is an oscillatory integral with smooth amplitude and can be analyzed using the standard theory of pseudo-differential operator. Each of the other four integrals concentrates on a part (a ray) of $\partial W_\Phi$. They can be analyzed using the common model introduced in Section \ref{S:Model}.

We will proceed the proof of Theorem \ref{T:Main1} as follows. In Section \ref{S:Model}, we introduce a family of model oscillatory integrals. We show that their twisted wavefront set belongs to the union of two intersecting Lagrangians, one is the diagonal. We also proceed to compute their principal or full symbol on these Lagrangians. In Section \ref{S:Proof}, we present the proof of Theorems \ref{T:Main1}, using our understanding in Section~\ref{S:Model}.

\begin{rem} In our subsequent papers \cite{Art-Sphere,Art-Sphere-Curved}, we adapt the technique developed in this article to study the artifacts in limited data problem of spherical mean transform, which arises in several imaging modalities (such as thermo/photo-acoustic tomography, ultrasound tomography). 
\end{rem}

\section{Model oscillatory integrals}\label{S:Model}
\noindent Let $m$ be a nonnegative real number and $\rho \in C^\infty(\rN^2 \setminus 0)$ be homogeneous of degree zero \footnote{A function $\rho \in C^\infty(\rN^2 \setminus 0)$ is homogeneous of degree $\ga$ if for all $\tau>0$, $\rho(\tau \xi)= \tau^\ga \rho(\xi)$.} and $\rho(\xi_1,.)$ is compactly supported for any $\xi_1 \in \rN$. We consider the oscillatory integral 
\begin{equation} \label{E:mupm} \mu_\pm(x,y) =\frac{1}{(2\pi)^2} \intl_{\rN^2} e^{i (x-y) \cdot \xi } |\xi|^m \, \rho(\xi) \, \bH(\pm \xi_2) \, d \xi.\end{equation}
Here, $\bH$ is the Heaviside function, defined by
\begin{eqnarray*}
\bH(s) = \left\{\begin{array}{l}1,\quad s \geq 0, \\[3 pt] 0, \quad s<0. \end{array} \right.
\end{eqnarray*}
\medskip

\noindent We also recall the diagonal canonical relation in $(\cT^* \rN^2 \setminus 0) \times (\cT^* \rN^2 \setminus 0)$ $$\Delta = \{(x,\xi; x, \xi): (x,\xi) \in \cT^*\rN^{2} \setminus 0\},$$
and define $\mC \subset (\cT^* \rN^2 \setminus 0) \times (\cT^* \rN^2 \setminus 0)$ by
$$\mC = \{(x,\xi; y, \xi) \in (\cT^* \rN^2 \setminus 0) \times (\cT^* \rN^2 \setminus 0): x_1-y_1=0,\xi_2=0\}.$$

\medskip


\begin{prop} \label{P:mu} We have
\begin{equation} \label{E:wave-mu} WF(\mu_\pm)' \subset \Delta \cup \mC.\end{equation}
Furthermore,
\begin{itemize}
\item[a)] Near $\Delta \setminus  \mC$, $\mu_\pm$ is microlocally in $I^{m}(\Delta)$ with the full symbol  \footnote{From now on, we will drop the term ``for large $|\xi|$'' for the sake of convenience.}:
$$\sg(x,\xi) \sim |\xi|^m \, \rho(\xi) \, \bH(\pm \xi_2).$$

\medskip
\item[b)] Assume that $\rho(\xi)$ vanishes to order $k$ at $\xi_2 = 0$ for any fixed $\xi_1 \neq 0$. Then, near $\mC \setminus \Delta$, $\mu_\pm$ is microlocally in the space $I^{m-k-\frac{1}{2}}(\mC)$  with the principal symbol
\begin{eqnarray} \label{E:sg0} \sg_0(x,y, \xi_1) =\left\{\begin{array}{ll} \frac{\pm 1}{\sqrt{2 \pi}} \frac{1}{[ i(y_2-x_2)]^{k+1}} \, \varphi_+^{(k)}(0) \,  \xi_1^{m-k},&\hskip  0 pt \xi_1>0, \\[6 pt] \frac{\pm 1}{\sqrt{2 \pi}} \frac{1}{[i(y_2-x_2)]^{k+1}} \, \varphi_-^{(k)}(0) \,  |\xi_1|^{m-k}, &\hskip 0 pt \xi_1<0, \end{array} \right.
\end{eqnarray} given the phase function $\phi(x,y,\xi_1) = (x_1-y_1) \xi_1$. Here, $\varphi_\pm \in C^\infty(\rN)$ is defined by the formula
\begin{equation} \label{E:phip} \varphi_\pm \big(\frac{\xi_2}{|\xi_1|}\big) = \rho(\xi),\quad \mbox{ for all } \pm \xi_1>0 \mbox{ and } \xi_2 \in \rN.\end{equation}
\end{itemize}
\end{prop}

Before proving the Proposition \ref{P:mu}, we would like to point out that the function $\varphi_\pm \in C^\infty(\rN)$ in (\ref{E:phip}) is well-defined since $\rho$ is homogeneous of degree zero. Moreover, since $\rho$ vanishes to order $k$ at $\xi_2=0$, $\varphi_\pm(\tau)$ also vanish to order $k$ at $\tau=0$.

\begin{proof} We only need to prove the proposition for $\mu_+$. The proof for $\mu_-$ is similar.

\medskip

\noindent {\bf Proof for (\ref{E:wave-mu}).} We follow the approach in \cite{FQ13}. Let $k$ be defined by
\begin{equation}\label{E:k} k(x) =\frac{1}{(2\pi)^2} \intl_{\rN^2} e^{i x \cdot \xi } |\xi|^m \, \rho(\xi) \, \bH(\xi_2) \, d \xi.\end{equation}
Then,
$$\mu_+(x,y) =k(x-y).$$
Therefore, see \cite[page 270]{Ho1-old},
\begin{equation} \label{E:mf}  WF(\mu_+) \subset \{(x,\xi;y,-\xi): (x-y,\xi) \in WF(k)\}.\end{equation}

\medskip

Due to (\ref{E:k}), up to a constant multiple, the Fourier transform of $k$ is $$K(\xi) = |\xi|^m \, \rho(\xi) \,\bH(\xi_2),$$
which is a homogenous distribution with wave front set $$WF(K) \subset \{(\xi,x): \xi_2=0, x_1=0\} \cup \{(0,x): x \in \rN^2\}.$$
We recall the following rule for wave front set of homogeneous distribution (see \cite[Theorem  8.1.4]{Ho1-old}):
\begin{eqnarray*} (x,\xi) \in WF(k) & \Longleftrightarrow & (\xi,-x) \in WF(K),  \quad \mbox{ if } \xi \neq 0 \mbox{ and } x \neq 0,\\
(0,\xi) \in WF(k) & \Longleftrightarrow & \xi  \in \supp(K),  \quad \mbox{ if } \xi \neq 0.
\end{eqnarray*}
Therefore, \begin{eqnarray} \label{E:wk}  WF(k) \subset \{(x,\xi): x_1=0, \xi_2 =0\} \cup \{(0,\xi): \xi \in \supp(K)\} .\end{eqnarray}

\medskip

Combining (\ref{E:mf}) and (\ref{E:wk}), arrive to 
\begin{multline*}  WF(\mu_+) \subset \{(x,\xi; y, -\xi): x_1-y_1=0, \xi_2 =0\} \\ \cup \{(x,\xi;y,-\xi): x-y=0,~\xi \in \supp(K) \}.\end{multline*}
That is, 
\begin{equation}\label{E:wfu}  WF(\mu)' \subset \mC \cup \{(x,\xi;x,\xi): ~\xi \in \supp(K) \}.\end{equation}
In particular, this implies
\begin{eqnarray*}  WF(\mu)' \subset \mC \cup \Delta.\end{eqnarray*}
We have finished the proof for (\ref{E:wave-mu}). 

\medskip

\noindent{\bf Proof for a).} Let $(x^*,\xi^*;x^*,\xi^*) \in \Delta \setminus \mC$, then $\xi^*_2 \neq 0$. Let $\Cf \in C^\infty(\rN^2 \setminus 0)$ be homogeneous of degree zero such that $\Cf(\xi)=1$ in an open cone $V_0$ containing $\xi^*$ and $\Cf(\xi) = 0$ in a conic neighborhood of the set $\{\xi: \xi_2 =0, \xi_1 \neq 0\}$.  We define
\begin{equation*}  \mu_*(x,y) =\frac{1}{(2\pi)^2} \intl_{\rN^2} e^{i (x-y) \cdot \xi } |\xi|^m \, \rho(\xi) \,\Cf(\xi) \, \bH(\xi_2) \, d \xi.\end{equation*}
We observe that the function
$$b(x,\xi)= |\xi|^m \, \rho(\xi) \,\Cf(\xi) \, \bH(\xi_2)$$ satisfies $b(x,\xi)  \in S^{m} (\rN^2 \times \rN^2)$ for lage $|\xi|$ and $\partial_x^\ag b(x,\xi)$ is locally integrable with respect to $\xi$ for any multi-index $\ag$ and any $x \in \rN^2$. Therefore, $\mu_* \in I^m(\Delta)$ with the symbol $\sg^* \sim b$ (see Lemma~\ref{L:PDO}). In particular, in the conic neighborhood $\rN^2 \times V_0$ of $(x^*,\xi^*)$, $$\sg^*(x,\xi) \sim |\xi|^m \, \rho(\xi) \, \bH(\xi_2).$$
It, therefore, suffices to prove (see Definition~\ref{D:PDO}): $$(x^*,\xi^*; x^*,\xi^*) \not \in WF(\mu_+ - \mu_*)'.$$

Indeed, similarly to (\ref{E:wfu}), we obtain
\begin{equation*}  WF(\mu -\mu_*)' \subset \mC \\ \cup \{(x,\xi;x,\xi): ~\xi \in \supp(K_0) \}.\end{equation*}
Here, 
$$K_0(\xi) = |\xi|^m \,\big[1-\Cf(\xi)\big] \, \rho(\xi) \,\bH(\xi_2)$$
is, up to a constant multiple, the Fourier transform of $(\mu - \mu_*)$.

\medskip

\noindent Since $\xi_2^* \neq 0$, one easily sees $(x^*,\xi^*;y^*,\xi^*) \not \in \mC$. Moreover, since $1-\Cf(\xi) =0$ in a neighborhood $V_0$ of $\xi^*$, we have
$$ (x^*,\xi^*;x^*,\xi^*) \not \in  \{(x,\xi;x,\xi): ~\xi \in \supp(K_0) \}.$$
Therefore, $$(x^*,\xi^*;x^*,\xi^*) \not \in WF(\mu_+ - \mu_*)'.$$
This finishes the proof for a).

\medskip

\noindent{\bf Proof for b).} We now analyze $\mu$ on $\mC \setminus \Delta$. Let $(x^*, \xi^*;  y^*, \xi^*) \in \mC \setminus \Delta$, then $x^*_2 \neq y^*_2$. Let $\mO \subset \rN^2 \times \rN^2$ be an open set containing $(x^*,y^*)$ such that $x_2 \neq y_2$ for any $(x,y) \in \mO$. It suffices to prove that $\mu_+|_{\mO} $ is in $I^{m-k-\frac{1}{2}}(\mC)$ with the stated symbol (see Definition~\ref{D:FIO}). 

\medskip

For $(x,y) \in \mO$, let us write
\begin{equation} \label{E:muc} \mu_+(x,y) = \frac{1}{(2 \pi)^2}  \intl_{\rN} e^{i (x_1-y_1) \xi_1 }  \, a(x,y,\xi_1) \,d \xi_1,\end{equation} where
\begin{equation} \label{E:a} a(x,y,\xi_1) = \intl_{\rN} e^{i (x_2-y_2) \xi_2 } \, |\xi|^m \, \rho(\xi) \, \bH(\xi_2) \, d\xi_2.\end{equation}
Since $\rho(\xi_1,.)$ is compactly supported, the integral on the right hand side of (\ref{E:a}) is, in fact, over a finite interval. Therefore, one can easily see that $a \in C^\infty((\rN^2 \times \rN^2) \times (\rN \setminus 0))$. We can also observe that $\pdh_x^\ag \pdh_y^\bg a(x,y,\xi_1)$ is locally integrable with respect to $\xi_1$ for any $(x,y) \in \rN^2 \times \rN^2$ and any multi-indices $\ag,\bg$. Due to Lemma~\ref{L:FIO}  \footnote{In this situation, $\ve = (0,1)$.}, it suffices to show that $a(x,y,\xi_1) \in S^{m-k}(\mO \times \rN)$ for large $|\xi_1|$ and its leading term is
\begin{eqnarray} \label{E:amk} a_{m-k}(x,y, \xi_1) =\left\{\begin{array}{ll} \frac{1}{[ i(y_2-x_2)]^{k+1}} \, \varphi_+^{(k)}(0) \,  \xi_1^{m-k},&\hskip  0 pt \xi_1>0, \\[6 pt] \frac{1}{[i(y_2-x_2)]^{k+1}} \, \varphi_-^{(k)}(0) \,  |\xi_1|^{m-k}, &\hskip 0 pt \xi_1<0, \end{array} \right.
\end{eqnarray}
\medskip

Indeed, let us consider $\xi_1>0$. Using the change of variable $\xi_2= \xi_1 \tau$, we obtain
\begin{eqnarray*} 
a(x,y,\xi_1) &=&  \xi_1^{m+1} \intl_0^\infty e^{i \, (x_2-y_2) \, \xi_1 \,\tau } \,(1+\tau^2)^{m/2} \,  \rho(\xi_1,\tau \, \xi_1) \,  d \, \tau. \end{eqnarray*}
Recalling that $\varphi_+: \rN \to \rN$ is defined by
\begin{equation*}  \varphi_+\big(\frac{\xi_2}{\xi_1}\big) = \rho(\xi),\quad \mbox{ for all } \xi_1>0 \mbox{ and } \xi_2 \in \rN,\end{equation*}
we arrive to
\begin{eqnarray*}
a(x,y,\xi_1) =  \xi_1^{m+1} \intl_0^\infty e^{i \, (x_2-y_2) \, \xi_1 \,\tau } \,\psi_+(\tau) \,  d \, \tau, \mbox{ for all } \xi_1>0,
\end{eqnarray*}
where $\psi_+(\tau) = (1+\tau^2)^{m/2} \,  \varphi_+(\tau).$
Since  $x_2 \neq y_2$ for all $(x,y) \in \mO$, we can write:
\begin{eqnarray*} 
a(x,y,\xi_1) =  \frac{1}{i(x_2-y_2)} \, \xi_1^m \intl_0^\infty (e^{i \, (x_2-y_2) \, \xi_1 \,\tau })_\tau \, \psi_+(\tau) \, d \, \tau.
\end{eqnarray*}
Taking integration by parts and noticing that $\psi_+$ is compactly supported, we obtain
\begin{eqnarray*} 
a(x,y,\xi_1) =  \frac{1}{i(y_2-x_2)} \, \xi_1^m \Big( \psi_+(0) +  \intl_0^\infty e^{i \, (x_2-y_2) \, \xi_1 \,\tau }  \, \psi_+'(\tau) \,  d \, \tau \Big).
\end{eqnarray*}
Continuing the successive integration by parts, we arrive to
\begin{eqnarray*} \mbox{} \quad \quad
a(x,y,\xi_1) =  \sum_{l=0}^{k+1} \frac{1}{[i(y_2-x_2)]^{l+1}} \, \psi_+^{(l)}(0) \,  \xi_1^{m-l} + R_{+}(x,y,\xi_1), \mbox{ for } \xi_1 >0,
\end{eqnarray*}
where
$$R_{+} (x,y,\xi_1) = \frac{1}{[i(y_2-x_2)]^{k+2}} \, \xi_1^{m-k-1} \intl_0^\infty e^{i \, (x_2-y_2) \, \xi_1 \,\tau }  \, \psi_+^{(k+2)}(\tau) \,  d \, \tau.$$
Since $\psi_+^{(l)}(0)=0$ for all $0 \leq l \leq k-1$, we obtain
\begin{multline} \label{E:R1}
a(x,y,\xi_1) =   \frac{1}{[i(x_2-y_2)]^{k+1}} \, \psi_+^{(k)}(0) \,  |\xi_1|^{m-k} \\ + \frac{1}{[i(x_2-y_2)]^{k+2}} \, \psi_+^{(k+1)}(0) \,  |\xi_1|^{m-k-1}  + R_{+}(x,y,\xi_1),\quad \xi_1>0.
\end{multline}
Using the same integration by parts technique as above, one can easily show that the function
$$r_{+}(x,y,\xi_1) = \intl_0^\infty e^{i \, (x_2-y_2) \, \xi_1 \,\tau }  \, \psi_+^{(k+2)}(\tau) \,  d \, \tau$$
satisfies $r_{+}(x,y,\xi_1) \in S^{0}(\mO \times \rN_+)$ for large $|\xi_1|$. Therefore, 
$$R_{+}(x,y,\xi_1)\in S^{m-k-1}(\mO \times \rN_+), \mbox{ for large } |\xi_1|.$$

\medskip

Similarly 
\begin{multline}\label{E:R2}\mbox{} \quad
a(x,y,\xi_1) =  \frac{1}{[i(x_2-y_2)]^{k+1}} \, \psi_-^{(k)}(0) \,  |\xi_1|^{m-k} \\ + \frac{1}{[i(x_2-y_2)]^{k+2}} \, \psi_-^{(k+1)}(0) \,  |\xi_1|^{m-k-1} + R_{-}(x,y,\xi_1),\quad \xi_1<0. 
\end{multline}
where $\psi_-(\tau) = (1+\tau^2)^{m/2} \,  \varphi_-(\tau)$ and
$$R_{-}(x,y,\xi_1)\in S^{m-k-1}(\mO \times \rN_+), \mbox{ for large } |\xi_1|.$$

\medskip

Therefore, from (\ref{E:R1}) and (\ref{E:R2}), $a(x,y,\xi_1) \in S^{m-k}(\mO \times \rN)$ for large $|\xi_1|$, with the top order term
\begin{eqnarray*} a_{m-k}(x,y, \xi_1) =\left\{\begin{array}{ll} \frac{1}{[i (y_2-x_2)]^{k+1}} \, \psi_+^{(k)}(0) \,  \xi_1^{m-k},&\hskip  0 pt \xi_1>0, \\[6 pt] \frac{1}{[i(y_2-x_2)]^{k+1}} \, \psi_-^{(k)}(0) \,  |\xi_1|^{m-k}, &\hskip 0 pt \xi_1<0. \end{array} \right.
\end{eqnarray*}
From the definition of $\psi_\pm$, it is easy to show that $a_{m-k}$ satisfies (\ref{E:amk}). This finishes our proof.
\end{proof}

\medskip

\noindent We now consider a generalization of $\mu_\pm$. Namely, let $\ve \in \rN^2$ be a unit vector and let us consider the distributions
\begin{equation} \label{E:mue} \mu_{\pm \ve} (x,y) = \frac{1}{(2 \pi)^2}\intl_{\rN^2} e^{i (x-y) \cdot \xi} |\xi|^m \, \rho_{\ve} (\xi) \, \bH(\pm \ve^\perp \cdot \xi) \, d \xi.\end{equation}
Let
$$\mC_\ve= \{(x,s \, \ve; x+ t \, \ve^\perp , s \, \ve): x \in \rN^2, \, t, s \in \rN, \, \ga \neq 0 \}. $$
\begin{prop}\label{P:Rot}
We have $$WF(\mu_{\pm \ve})' \subset \Delta \cup \mC_\ve.$$ 
\begin{itemize}
\item[a)] Near $\Delta \setminus \mC_\ve$, $\mu_{\pm \ve}$ is in the space $I^{m}(\Delta)$ with the symbol:
$$\sg(x,\xi) \sim |\xi|^m \, \rho_{\ve} (\xi) \, \bH(\pm \ve^\perp \cdot \xi). $$ 
\item[b)] Assume that $\rho_\ve$ vanishes to order $k$ on the line $\ve^\perp \cdot \xi =0$. Then, near $\mC_\ve \setminus \Delta$, $\mu_{\pm \ve}$ is microlocally in the space $I^{m-k-1/2}(\mC_\ve)$. Moreover, given the phase function $\phi(x,y,\tau) = (x-y) \cdot \ve \, \tau$,  its principal symbol is given by
\begin{eqnarray*} \sg_0(x, x+ t \, \ve^\perp, \tau) = \left\{\begin{array}{ll} \pm  \frac{1}{\sqrt{2 \pi}}\frac{1}{(i \, t)^{k+1}} \, \varphi_{+}^{(k)}(0) \, \tau^{m-k},&\hskip 0 pt \tau>0, \\[6 pt] \pm \frac{1}{\sqrt{2 \pi}}\frac{1}{( i \, t)^{k+1}} \, \varphi_{-}^{(k)}(0) \, |\tau|^{m-k},&\hskip 0 pt \tau<0. \end{array} \right.
\end{eqnarray*}
Here, $\varphi_\pm: \rN \to \rN$ satisfies
$$\varphi_\pm \big(\frac{\xi \cdot \ve^\perp}{|\xi \cdot \ve|}\big)= \rho_\ve(\xi) ,\quad \mbox{ for all $\xi$ such that } \pm \xi \cdot \ve >0.$$ 
\end{itemize}
\end{prop}

\begin{proof}
The Proposition \ref{P:Rot} can be obtained from Propositions~\ref{P:mu} by a simple change of variables. Indeed, let $$\eta = (\ve \cdot \xi, \ve^\perp \cdot \xi), \quad x'= (\ve \cdot x,\ve^\perp \cdot x), \quad y' = (\ve \cdot y,\ve^\perp \cdot y).$$
and $\rho,\mu_\pm$ be defined by
$$\mu_{\pm \ve} (x,y) = \mu_\pm(x',y'), \quad \rho_\ve(\xi)= \rho(\eta).$$
By changing the variables in (\ref{E:mue}), we obtain
$$\mu_\pm(x',y') = \frac{1}{(2 \pi)^2}\intl_{\rN^2} e^{i (x'-y') \cdot \eta} |\eta|^m \, \rho (\eta) \, \bH(\pm \eta_2) \, d \eta.$$
Applying Propositions \ref{P:mu} for $\mu_\pm$ and translating the result back to $\mu_{\pm \ve}$, we finish the proof.
\end{proof}

\section{Proof of Theorem \ref{T:Main1}} \label{S:Proof}
\noindent Let us first divide the boundary of $W_\Phi$ into four rays. Namely, let $$\ve_1=-\ve_3= (\cos \Phi, \sin \Phi), \quad \ve_2=-\ve_4= (\cos \Phi, - \sin \Phi).$$
and for $j=1,\dots,4$:
\begin{eqnarray*} R_j &=& \{\xi: \, \xi= r\, \ve_j,~ r  > 0 \}. \end{eqnarray*}
It is obvious that $$R_1 \cup R_2 \cup R_3 \cup R_4 = \pdh W_{\Phi} \setminus \{0\}. $$
For $j=1,\dots, 4$, let $\rho_j \in C^\infty(\rN^2 \setminus 0)$ be homogeneous of degree zero such that $\rho_j =1$ in a (small) conic neighborhood of $R_j$. Moreover, $\rho_j$ is supported inside a small conic neighborhood of $R_j$ \footnote{This, in particular, implies $\supp (\rho_j) \cap \supp (\rho_k) =\{0\}$, for $j \neq k$.}. 

We can write:
\begin{eqnarray*} 
 \mu_m(x,y) &=& \frac{1}{(2 \pi)^2}\intl_{\rN^2} e^{i (x-y) \cdot \xi } |\xi|^m \, \kappa (\xi/|\xi|)\, \chi_{\Phi}(\xi) \, d \xi \\ &=&  \frac{1}{(2 \pi)^2}\intl_{\rN^2} e^{i (x-y) \cdot \xi } |\xi|^m \, \Big[1- \sum_{j=1}^4 \rho_j(\xi) \Big] \, \kappa (\xi/|\xi|)\, \chi_{\Phi}(\xi) \, d \xi \\ &+& \sum_{j=1}^4 \frac{1}{(2 \pi)^2}\intl_{\rN^2} e^{i (x-y) \cdot \xi } |\xi|^m \,  \rho_j(\xi) \, \kappa (\xi/|\xi|)\, \chi_{\Phi}(\xi) \, d \xi \\&=& \mu^0(x,y) + \sum_{j=1}^4 \mu^j(x,y).
\end{eqnarray*}

\noindent\underline{\bf Properties of $\mu^0$:} We notice that the function 
$$b(x,\xi) = |\xi|^m \, \Big[1- \sum_{j=1}^4 \rho_j(\xi) \Big] \, \kappa (\xi/|\xi|)\, \chi_{\Phi}(\xi)$$
is smooth in $\rN^2 \times (\rN^2 \setminus 0)$. It is, hence, clear that $b(x,\xi)$ satisfies the conditions in Lemma~\ref{L:PDO}. Therefore, $\mu^0 \in I^m(\Delta)$ with the symbol $\sg^0(x,\xi) \sim b(x,\xi)$. This, in particular, implies $$WF(\mu^0)' \subset \Delta.$$

\medskip

\noindent\underline{\bf Properties of $\mu^1$:} Since $\rho_1$ is supported in a small conic neighborhood of $R_1$, we have $$\chi_{\Phi}(\xi) = \bH(-\ve_1^\perp \cdot \xi), \quad \mbox{ for all } \xi \in \supp(\rho_1).$$
Therefore,
\begin{eqnarray*}
\mu^1(x,y) &=& \frac{1}{(2 \pi)^2}\intl_{\rN^2} e^{i (x-y) \cdot \xi } |\xi|^m \,  \kappa (\xi/|\xi|)\, \rho_1(\xi) \, \chi_{\Phi}(\xi) \, d \xi 
\\ &=& \frac{1}{(2 \pi)^2}\intl_{\rN^2} e^{i (x-y) \cdot \xi } |\xi|^m \, \kappa (\xi/|\xi|) \, \rho_1(\xi) \, \bH(- \ve_1^\perp \cdot \xi) \, d \xi.
\end{eqnarray*}
Applying Proposition~\ref{P:Rot}, we obtain $$WF(\mu^1)' \subset \Delta \cup \mC_1.$$
Moreover, due to Proposition~\ref{P:Rot}~a), near $\Delta \setminus \mC_1$, $\mu^1$ is microlocally in the space $I^m(\Delta)$ with the full symbol
\begin{eqnarray*} \sg^{1}(x,\xi) \sim  |\xi|^m \, \kappa (\xi/|\xi|) \, \rho_1(\xi) \, \chi_{\Phi}(\xi) .\end{eqnarray*}

\medskip

On the other hand, due to Proposition~\ref{P:Rot}~b), near $\mC_1 \setminus \Delta$, $\mu^1$ is microlcally in the space $I^{m-k-\frac{1}{2}}(\mC_1)$ with the principal symbol 
\begin{eqnarray} \label{E:psymb1} \sg^1_0(x, x+ t \, \ve_1^\perp, \tau) = \left\{\begin{array}{ll} - \frac{1}{\sqrt{2 \pi}}\frac{1}{(i \, t)^{k+1}} \, \varphi_{+}^{(k)}(0) \, \tau^{m-k},&\hskip 0 pt \tau>0, \\[6 pt] - \frac{1}{\sqrt{2 \pi}}\frac{1}{(i \, t)^{k+1}} \, \varphi_{-}^{(k)}(0) \, |\tau|^{m-k},&\hskip 0 pt \tau<0. \end{array} \right.
\end{eqnarray} given the phase function $\phi_1(x,y,\tau) = (x-y) \cdot \ve_1 \, \tau.$
Here, $\varphi_+$ is defined by
$$\varphi_\pm \big(\frac{\ve_1^\perp \cdot \xi}{|\ve_1 \cdot \xi|} \big) = \kappa(\frac{\xi}{|\xi|} \big) \, \rho_1(\xi) ,~\mbox{ for all } \pm \ve_1 \cdot \xi >0. $$
Direct calculations show that 
\begin{equation*} 
\varphi_-^{(k)}(0) = 0, \quad \varphi_+^{(k)}(0) = \varkappa^{(k)}(\Phi),
\end{equation*} where, we recall, $\varkappa$ is defined by $\varkappa(\phi) = \kappa(\cos \phi, \sin \phi)$. 
Therefore,
\begin{eqnarray*} \sg^1_0(x, x+ t \, \ve_1^\perp, \tau) = \left\{\begin{array}{ll} - \frac{1}{\sqrt{2 \pi}}\frac{\varkappa^{(k)}(\Phi)}{(i \, t)^{k+1}} \,  \tau^{m-k},&\hskip 0 pt \tau>0, \\[6 pt] 0. \end{array} \right.
\end{eqnarray*}

\medskip

\noindent\underline{\bf Properties of $\mu^3$:} Similarly to $\mu^1$, we obtain
$$WF(\mu^3)' \subset \Delta \cup \mC_1.$$

\medskip

\noindent Moreover, near $\Delta \setminus \mC_1$, $\mu^3$ is microlocally in the space $I^m(\Delta)$ with the full symbol
\begin{eqnarray*} \sg^{3}(x,\xi) \sim  |\xi|^m \, \kappa (\xi/|\xi|) \, \rho_3(\xi) \, \chi_{\Phi}(\xi) .\end{eqnarray*}
On the other hand, near $\mC_1 \setminus \Delta$, $\mu^3$ is microlcally in the space $I^{m-k-\frac{1}{2}}(\mC_1)$ with the principal symbol 
\begin{eqnarray*} \sg^3_0(x, x+ t \, \ve_1^\perp, \tau) = \left\{\begin{array}{ll} 0, \\ \frac{1}{\sqrt{2 \pi}}\frac{(-1)^k \varkappa^{(k)}(\Phi)}{(i \, t)^{k+1}} \,  |\tau|^{m-k},&\hskip 0 pt \tau>0, \end{array} \right.
\end{eqnarray*} given the phase function $\phi_1(x,y,\tau) = (x-y) \cdot \ve_1 \, \tau.$

\medskip

\noindent \underline{\bf Properties of $\mu^{1,3}=\mu^1+\mu^3$:} Adding up the above results for $\mu^1$ and $\mu^3$, we obtain for $\mu^{1,3} = \mu^1+ \mu^3$:  $$WF(\mu^{1,3})' \subset \Delta \cup \mC_1,$$
Moreover, 
\begin{itemize}
\item[i)] Near $\Delta \setminus \mC_1$, $\mu^{1,3}$ is microlocally in the space $I^m(\Delta)$ with the full symbol
\begin{eqnarray*} \sg^{1,3}(x,\xi) \sim  |\xi|^m \, \kappa (\xi/|\xi|) \, \big[ \rho_1(\xi) + \rho_3(\xi) \big] \chi_{\Phi}(\xi) .\end{eqnarray*}
\item[ii)] Near $\mC_1 \setminus \Delta$, $\mu^{1,3}$ is microlcally in the space $I^{m-k-\frac{1}{2}}(\mC_1)$ with the principal symbol 
\begin{eqnarray*}
\sg_0(x, x+ t \, \ve_1^\perp, \tau) = - \frac{1}
{\sqrt{2\pi}} \frac{\varkappa^{(k)} (\Phi)}{[\sgn (\tau)\, i \, t]^{k+1}}  \, |\tau|^{m-k},
\end{eqnarray*} given the phase function $\phi_1(x,y,\tau) = (x-y) \cdot \ve_1 \, \tau.$
\end{itemize}


\medskip

\noindent \underline{\bf Properties of $\mu^{2,4} = \mu^2 + \mu^4$:} Similarly to $\mu^{1,3}$, we obtain for $\mu^{2,4} = \mu^2 +\mu^4$: $$WF(\mu^{2,4})' \subset \Delta \cup \mC_2.$$
Moreover, 
\begin{itemize}
\item[i)] Near $\Delta \setminus \mC_2$, $\mu^{2,4}$ is microlocally in the space $I^m(\Delta)$ with the full symbol
\begin{eqnarray*} \sg^{2,4}(x,\xi) \sim  |\xi|^m \, \kappa (\xi/|\xi|) \, \big[ \rho_2(\xi) + \rho_4(\xi) \big] \chi_{\Phi}(\xi) .\end{eqnarray*}

\item[ii)] Near $\mC_2\setminus \Delta$, $\mu^{2,4}$ is microlocally in the space $I^{m-k-\frac{1}{2}}(\mC_2)$ with the principal symbol
\begin{eqnarray*}
\sg_0(x, x+ t \, \ve_2^\perp, \tau) = \frac{1}{\sqrt{2\pi}}\frac{\varkappa^{(k)} (- \Phi)}{[i \, \sgn (\tau) \, t]^{k+1} } \,  |\tau|^{m-k},
\end{eqnarray*} given the phase function $\phi_2(x,y,\tau) = (x-y) \cdot \ve_2 \, \tau$. 
\end{itemize}
\medskip

\noindent \underline{\bf Properties of $\mu_m$:} Adding up all the above characterizations of $\mu^0, \mu^{1,3},\mu^{2,4}$, we obtain Theorem~\ref{T:Main1}~a) \& b), and the inclusion \footnote{Here, we have used the inclusion $WF(\mu^1 + \mu^2) \subset WF(\mu^1) \cup WF(\mu^2)$.} \begin{equation*}  WF(\mu_m)' \subset \Delta \cup \mC_1 \cup \mC_2.\end{equation*}
Since $\sg(x,\xi) \sim 0$ if $\xi \not \in \cl(W_\Phi)$, we arrive to (see, e.g., \cite[Proposition 2.5.7]{hormander71fourier}):
\begin{equation*} WF(\mu) \subset \Delta_\Phi \cup \mC_1 \cup \mC_2.\end{equation*}
This finishes our proof.
\section*{Acknowledgment}

\noindent The author is grateful to Professor G. Uhlmann for introducing him to the theory of pseudo-differential operators with singular symbol, which inspires the main idea of the article. He is thankful to Professor T. Quinto for his kind attention to this work and the information about the closely related work by Professor A. Katsevich \cite{Kat-JMAA}. Helpful discussions with Dr J. Frikel (tomography) and Dr S. Eswarathasan (microlocal analysis) are also appreciated. 

\medskip

The author heartily thanks the anonymous referees for their comments and suggestions, which significantly helped improve the writing of this article.


\newcommand{\etalchar}[1]{$^{#1}$}
\def\dbar{\leavevmode\hbox to 0pt{\hskip.2ex \accent"16\hss}d}

\end{document}